\documentclass[draft]{amsart}
\usepackage{amssymb} 
\usepackage{amsmath}
\newtheorem{theorem}{Theorem}[section]
\newtheorem{lemma}[theorem]{Lemma}
\newtheorem{proposition}[theorem]{Proposition}

\theoremstyle{definition}
\newtheorem{definition}[theorem]{Definition}

\newtheorem{remark}[theorem]{Remark}
\theoremstyle{approach}

\newcommand{\NN}{\mathbb N}
\newcommand{\CC}{\mathbb C}

\numberwithin{equation}{section}

\begin{document}
\setcounter{page}{1}

\title[The {\text BSE} Property for commutative Fr\'echet Algebras]{The Bochner-Schoenberg-Eberlein Property for commutative Fr\'echet Algebras}
\author[M. Amiri and  A. Rejali]{M. Amiri and A. Rejali}
\subjclass[2010]{46H05, 46H25, 46J05, 46J20}
\keywords{{\text BSE}-algebra, {\text BSE}-function, commutative Fr\'echet algebra, multiplier algebra.}

\begin{abstract}

A class of commutative Banach algebras which satisfy a Bochner-
Schoenberg-Eberlein-type inequality was introduced by Takahasi 
and Hatori. We generalize this property for the commutative Fr\'echet algebra
 $(\mathcal A, p_\ell)$. Furthermore, some of the main results in the class of Banach algebras, 
will be verified and generalized for the Fr\'echet case. 

\end{abstract}
 
\maketitle \setcounter{section}{-1}

\section{Introduction}

Let $\mathcal A$ be a commutative semisimple Banach algebra with
 Gelfand spectrum $\Delta(\mathcal A)$. In addition, suppose that 
 $C_b(\Delta(\mathcal A))$ is the space consisting of all continuous and bounded complex-valued  
functions on $\Delta(\mathcal A)$. It is known that
$C_b(\Delta(\mathcal A))$ is a Banach algebra under supremum norm,
 defined as
$$\|f\|_\infty=\sup\big{\{}|f(\varphi)|:\;\varphi\in\Delta(\mathcal
A)\big{\}}\;\;\;\;\big{(}f\in C_b(\Delta(\mathcal A))\big{)},$$
and the pointwise product. Takahasi and Hatori \cite{15} studied
commutative Banach algebras, satisfying a Bochner-Schoenberg-Eberlein-type theorem.
Following \cite{15},
$\sigma\in C_b(\Delta(\mathcal A))$ is called a {\text BSE}-function, if
there exists a positive real number $\beta$ such that for every
finite number of complex-numbers $c_1,\cdots,c_n$ and the same number of
$\varphi_1,\cdots,\varphi_n$ in $\Delta(\mathcal A)$ the
inequality
$$\vert\sum_{i=1}^{n}c_{i}\sigma (\varphi_{i})\vert\leq\beta
\Vert\sum_{i=1}^{n}c_{i}\varphi_{i}\Vert_{\mathcal A^{\ast}}$$
holds. The {\text BSE}-norm of $\sigma$, denoted by $\|\sigma\|_{\text{BSE}}$, is defined as the infimum of all such $\beta$. The set of all {\text BSE}-functions, denoted by
$C_{\text{BSE}}(\Delta(\mathcal A))$, under the norm $\|.\|_{\text{BSE}}$ is a
complete semisimple subalgebra of $C_b(\Delta (\mathcal A))$. Let
$M(\mathcal A)$ be the multiplier algebra of $\mathcal A$. Then for each
 $T\in M(\mathcal A)$,
there exists a unique bounded continuous function $\widehat{T}$ on
$\Delta(\mathcal A)$ where
$\widehat{T(a)}(\varphi)=\widehat{T}(\varphi)\varphi(a)$ for all
$a\in\mathcal A$ and $\varphi\in\Delta(\mathcal A)$; see
\cite{10}. Write $
\widehat{M(\mathcal A)}=\{\widehat{T}:\;T\in M(\mathcal A)\}.$ The Banach algebra 
$\mathcal A$ is called a {\text BSE}-algebra if
$\mathcal A$ satisfies the condition $\widehat{M(\mathcal
A)}=C_{\text{BSE}}(\Delta(\mathcal A))$. 
Takahasi and Hatori provided some important
necessary and sufficient conditions for a Banach algebra to be a
{\text BSE}-algebra; see \cite{15}. The abbreviation {\text BSE} stands for
 Bochner-Schoenberg-Eberlein and refers to a theorem, proved by Bochner and
Schoenberg for the additive group of real numbers \cite{1, 13}, and
then in general, by Eberlein \cite{2} for any locally compact abelian
group $G$. In this terminology, the group algebra $L^1(G)$ is
a {\text BSE}-algebra; see \cite{12}. Moreover, Kaniuth 
and \"Ulger \cite{9} have studied more general properties of {\text BSE}-algebras. 
Abtahi et al. in \cite{2A}
studied the $\mathcal A$-valued Lipschitz functions algebra
 $Lip_{\alpha}(X,\mathcal A)$ and showed that it is a {\text BSE}-algebra
 if only if $\mathcal A$ is so. 
 Abtahi and Kamali \cite{1A} also 
 studied the {\text BSE}
 property of the Banach algebra of $\mathcal A$-valued functions 
 $\ell_{p}(X,\mathcal A)$, for $1\leq p <\infty$. They proved that
 $\ell_{p}(X,\mathcal A)$ is a {\text BSE}-algebra if and only if $X$ is
 finite and $\mathcal A$ is a {\text BSE}-algebra. Several other authors also
 have investigated BSE-algebras; see for example \cite{4, 5, 6, 8, 16}.  

Following \cite{10}, a Fr\'echet space is a completely metrizable locally convex space where its 
topology is generated by translating invariant metric space. Moreover, following \cite{3},
a complete topological algebra $\mathcal{A}$ is a Fr\'echet algebra if its topology is produced by a countable family of
increasing submultiplicative seminorms $(p_\ell)_{\ell\in\mathbb{N}}$.
The class of 
Fr\'echet algebras which is an important class of locally convex algebras has
 been widely studied by many authors. For a full understanding of 
Fr\'echet algebras, one may refer to \cite{3, 11}. Note that 
every Banach algebra (space) is a Fr\'echet algebra (space), but in general 
the converse is not true. Some differences between Fr\'echet and Banach
 algebras surveyed in \cite{3A}. In this paper, we
 generalize the {\text BSE} property for the commutative semisimple
Fr\'echet algebra $(\mathcal A,p_\ell)_{\ell\in\NN}$. We introduce the space
$C_{\text{BSE}}(\Delta({\mathcal A}))$ with a family of continuous
seminorms $(r_\ell)_{\ell\in\NN}$ and show that $\big{(}C_{\text{BSE}}(\Delta({\mathcal
A})), r_\ell\big{)}_{\ell\in\NN}$ is a commutative semisimple Fr\'echet algebra. Moreover, 
we generalize some main results from Banach algebras to Fr\'echet
case. We define the notion of bounded 
$\Delta$-weak approximate identity for Fr\'echet algebras,
according to its definition for Banach algebras. Analogous to the Banach case,
 we prove that
$\widehat{M(\mathcal A)}\subseteq C_{\text{BSE}}(\Delta(\mathcal A))$ if
and only if $\mathcal A$ has a bounded $\Delta$-weak approximate identity.
 Let $(\mathcal A,p_\ell)$ be a commutative semisimple Fr\'echet algebra and let $\mathcal{A}^*$ denoted the topological dual of $\mathcal A$.
 We recall from \cite{11} that the strong topology on $\mathcal{A}^*$ is generated by seminorms $(P_M)$
 where  $M$ is a bounded set in
 $\mathcal A$.
Put
$$\beta_{M}({\mathcal A}):=\sup\big{\lbrace}\vert\sum_{i=1}^{n}c_{i}\vert:
\, P_{M}(\sum_{i=1}^{n} c_{i} \varphi_i)\leq 1 ,\;c_1
\cdots , c_n \in \mathbb{C},\, \varphi_1, \cdots,\varphi_n \in
\Delta({\mathcal A})\big{\rbrace}.$$ 
In the present paper, 
we show that
$C_{\text{BSE}}(\Delta({\mathcal A}))$ is unital if and only if
$\beta_{M}({\mathcal A})<\infty$, for some bounded $M$.
 
\section{Preliminaries}

In this section, we provide some basic definitions and frameworks, which will be required
 throughout the paper. 
 
 Let $E$ be a Hausdorff locally
convex space. Following \cite{11},
 $E$ has a fundamental system of seminorms
$(p_{\alpha})_{\alpha\in \Lambda}$, i.e. 
 a family of the
continuous seminorms satisfying the following properties:
\begin{enumerate}
\item[(i)] for every nonzero $x\in E$, there exists an
${\alpha}\in \Lambda$ with $p_{\alpha}(x)>0$; 
\item[(ii)] for all
${\alpha,\beta}\in \Lambda$, there exist ${\gamma}\in \Lambda$ and $K>0$ such
that $$\max\big{(}p_{\alpha}(x),p_{\beta}(x)\big{)}\leq K
p_{\gamma}(x)\;\;\;\;\;\;\;\;(x\in E).$$
\end{enumerate}
Throughout the paper, all locally convex spaces are assumed to be Hausdorff. 
Fr\'echet spaces are special locally convex spaces which we recalled in the previous section.
We also recalled the notion of Fr\'echet algebra.

Let $(\mathcal A,p_\ell)$ be a Fr\'echet algebra. 
A locally convex (resp. Fr\'echet) $\mathcal A$-bimodule is a locally convex (resp. Fr\'echet) space
$X$ together with the structure of an
$\mathcal A$-bimodule such that the corresponding mappings are
separately continuous.
Note that ${\mathcal A}^*$ is a locally convex
$\mathcal A$-bimodule with the module actions given by
$$
<a.f,b>=<f,ba>\;\;\;\text{and}\;\;\;<f.a,b>=<f,ab>\;\;\;\;\;\;\;\;\;\;\;(a,b\in\mathcal A,\;f\in{\mathcal A}^*),
$$
and the strong topology on bounded subsets of $\mathcal A$.
We know that
the net $(f_\alpha)_\alpha$ in ${\mathcal A}^*$ is convergent to
$f\in {\mathcal A}^*$ with respect to the strong topology,
if for every bounded subset $M$ of ${\mathcal A}$,
$$
\sup_{a\in M}|\langle f_\alpha-f,a\rangle|\longrightarrow_\alpha
0.
$$

It should be noted that, by \cite[page 79]{3-1}, if $\mathcal A$ is a
Fr\'echet algebra, then ${\mathcal A}^{**}$ is also a
Fr\'echet algebra with respect to the first Arens product and the strong topology on bounded subsets of ${\mathcal A}^*$.
Moreover, ${\mathcal A}^{**}$ is a Fr\'echet $\mathcal A$-bimodule via the 
module actions
$$<a.F,f>=<F,f.a>\;\;\;\text{and}\;\;\;<F.a,f>=<F,a.f>,$$ 
for all $a\in\mathcal
A$, $f\in {\mathcal A}^*$, and $F\in {\mathcal A}^{**}$.

Let us to recall that for a Fr\'echet algebra $\mathcal A$, the first Arens product on ${\mathcal
A}^{**}$ is defined by
$$
<F\square G , f>=<F , G.f>\;\;\;\;\;\;\;\;\;\;\;\;\;\;\;(f\in {\mathcal
A}^{*}),
$$
for all $F,G\in{\mathcal
A}^{**}$ where
$$
<G.f , a>=<G , f.a>\;\;\;\;\;\;\;\;\;\;\;\;\;\;(a\in\mathcal A).
$$

We now present some useful information about the second dual of the Fr\'echet algebra
$(\mathcal A,p_\ell)$. For each $\ell\in\Bbb N$, put $M_\ell:=\lbrace a\in
\mathcal A:\, p_\ell(a) < 1 \rbrace$.
Following \cite{11}, $(M_\ell)_{\ell\in\Bbb N}$ is a decreasing fundamental system of absolutely convex zero neighborhoods in $\mathcal A$. Consequently, $(M_{\ell}^{\circ})_{\ell\in\Bbb N}$ is an increasing
fundamental system of bounded sets in $\mathcal A^{\ast}$ where
 $M_{\ell}^{\circ}$ is the
polar of $M_{\ell}$, defined as
$$
M_{\ell}^{\circ}=\{f\in {\mathcal A}^{*}:\;|f(a)|\leq 1\;\text{for all}\;
a\in M_{\ell}\};
$$
see \cite[Lemma 25.5]{11} for details.
Consider the seminorms $(P_{M_{\ell}})_{\ell\in\Bbb N}$ and $(P_{M_{\ell}^{\circ}})_{\ell\in\Bbb N}$ respectively on $\mathcal A^{\ast}$ and ${\mathcal A}^{\ast \ast}$, where 
$$P_{M_{\ell}}(f):=\sup \lbrace \vert f(a)\vert:\,a \in M_{\ell}\rbrace\;\;\;\;\;\;\;\;\;\;\;\;(f\in\mathcal A^{\ast}),$$
and 
$$P_{M_{\ell}^{\circ}}(F):=\sup \lbrace \vert F(f)\vert:\,f \in M_{\ell}^{\circ}\rbrace\;\;\;\;\;\;\;\;\;\;\;\;(F\in{\mathcal A}^{\ast \ast}).$$
Now for each $\ell\in \mathbb{N}$, the following inequalities are immediate.
For all $a\in \mathcal A$, $f\in {\mathcal
A}^{\ast}$, and $F,G\in {\mathcal A}^{\ast \ast}$
we have
\begin{equation}\label{e1}
\vert f(a) \vert \leq  P_{M_\ell}(f)\; p_\ell(a),
\end{equation}
\begin{equation}\label{e2}
\vert G(f)\vert \leq
P_{M_{\ell}^{\circ}}(G)\; P_{M_\ell}(f), 
\end{equation}
\begin{equation}\label{e3}
P_{M_\ell}(f\cdot a)\leq P_{M_\ell}(f)\; p_l(a),
\end{equation}
\begin{equation}\label{e4}
P_{M_\ell}(G\cdot f)\leq P_{M_{\ell}^{0}}(G)\; P_{M_\ell}(f),
\end{equation}
and also
\begin{equation}\label{e5}
P_{M_{\ell}^{\circ}}(F\square G)\leq P_{M_{\ell}^{\circ}}(F)\; P_{M_{\ell}^{\circ}}(G).
\end{equation}
For convenience, we set $r_\ell(F):=P_{M_{\ell}^{\circ}}(F)$ for each $\ell\in\mathbb{N}$ and $F\in\mathcal{A}^{**}$.
By applying the last inequality, 
$(r_\ell)_{\ell\in\Bbb N}$ is an increasing sequence of submultiplicative 
seminorms on ${\mathcal A}^{\ast\ast}$.
Following \cite{3-1,11}, if
$({\mathcal A},p_\ell)$ is a Fr\'echet algebra, then $({\mathcal A}^{\ast \ast},r_\ell)=({\mathcal
A}^{\ast\ast},P_{N})$
where $N \subseteq {\mathcal A}^{\ast}$ runs through the bounded
subsets of ${\mathcal A}^{\ast}$. 
In particular, $({\mathcal A}^{\ast \ast},r_\ell)$ is complete.
Now, consider the canonical embedding $\iota:\mathcal A\rightarrow {\mathcal A}^{\ast \ast}$, 
$a\mapsto \widehat{a}$
where $\widehat{a}(f)=f(a)$ for each $f\in {\mathcal A}^{\ast}$. 
Following \cite{11}, $\iota$ is continuous.
Moreover,
$r_\ell(\widehat{a})=p_\ell(a)$ for 
each $a\in\mathcal A$ and $\ell\in \mathbb{N}$.
Indeed, by using
the general framework of the Hahn-Banach theorem
 \cite[Proposition 22.12]{11},
there exists
$f_0\in M_{\ell}^{\circ}$ such that $\vert f_0(a)\vert=p_\ell(a)$. Thus,
\begin{center}
$p_\ell(a)\leq\sup\lbrace\vert f(a)\vert:\,f \in M_{\ell}^{\circ} \rbrace
=r_\ell(\widehat{a}).$
\end{center}
On the other hand, for each $f\in M_{\ell}^{\circ}$ we have 
$$\vert f(a)\vert\leq
P_{M_{\ell}}(f)\;p_{\ell}(a)\leq p_{\ell}(a).$$
 For this reason, 
 $r_\ell(\widehat{a})=\underset{f\in M_{\ell}^{\circ}}{\sup}\lbrace\vert
f(a)\vert\rbrace\leq p_\ell(a)$ 
and so $r_\ell(\widehat{a})=p_\ell(a)$.
 
\section{{\text BSE}-functions for Fr\'echet algebras}
 
Let $(\mathcal A,p_\ell)$ be a commutative Fr\'echet algebra.
Recall from the literature that $\mathcal A$ is called without order 
if 
$
\big{\{}a: a\mathcal A=\{0\}\big{\}}=\{0\}.
$
Following \cite{3}, $\mathcal A$ is called semisimple if the Gelfand map
$\mathcal A\rightarrow C_b(\Delta(\mathcal A))$ given by
$$
a\mapsto\widehat{a},\;\;\;\widehat{a}(\varphi)=\varphi(a)\;\;\;\;\;\;\;\;\;\;\;\;\;(\varphi\in\Delta(\mathcal A))
$$
is injective.
In other words,
$
\bigcap_{\varphi\in\Delta(\mathcal A)}\text{ker}(\varphi)=\{0\}.
$
We now have the following result which is known in the Banach case. 
The proof is straightforward and it is therefore omitted.

\begin{lemma}\label{withoutorder}
Every commutative semisimple Fr\'echet algebra is without order.
\end{lemma}

\begin{definition}
Let $(\mathcal{A},p_{\ell})$ be a Fr\'echet algebra. A bounded complex-valued
continuous function $\sigma$ defined on $\Delta(\mathcal A)$ is called a
{\text BSE}-function if there exist a bounded set $M$ in $\mathcal A$ and 
 a positive real number $\beta_{M}$ such that 
 for every finite number of complex-numbers
 $c_1,\cdots, c_n$ and the same number of 
 $\varphi_1,\cdots, \varphi_n$ in $\Delta(\mathcal A)$ the inequality
$$
\vert\sum_{i=1}^{n} c_{i}\sigma(\varphi_{i})\vert\leq\beta_{M}
P_{M}(\sum_{i=1}^{n}c_{i}\varphi_{i})
$$
holds. The set of all {\text BSE}-functions is denoted by $C_{\text{BSE}}(\Delta(\mathcal A))$.
\end{definition}

Let $(\mathcal{A},p_{\ell})$ be a Fr\'echet algebra, and for every $F\in {\mathcal A}^{\ast\ast}$, consider the restriction map $\sigma=F\vert_{\Delta(\mathcal A)}$.
Set
$${r_{\ell}}(\sigma):=P_{M_{\ell}^{\circ}\cap<\Delta(\mathcal A)>}(F)\;\;\;\;\;\;\;\;\;\;\;\;\;(\ell\in\mathbb{N}).$$
Clearly, $C_{\text{BSE}}(\Delta(\mathcal A))\subseteq \ell_{\infty}(\Delta(\mathcal A))$. We also have
$$r_{\ell}(\sigma)=sup\big{\lbrace}\vert\sum_{i=1}^{n} c_{i}\sigma(\varphi_{i})\vert:
P_{M_{\ell}}(\sum_{i=1}^{n}c_{i}\varphi_{i})\leq 1\big{\rbrace}\;\;\;\;\;\;\;\;\;\;\;\;\;(\ell\in\mathbb{N}).$$
For each $\ell\in\mathbb{N}$,
set $\alpha_{\ell}:=sup\lbrace p_{\ell}(a): a\in M\rbrace$.
Note that for each $\ell\in\mathbb{N}$,
$\alpha_{\ell}<\infty$ since  
$M$ is bounded; see \cite[Remark 23.2]{11}.
Hence, by applying the inequality (\ref{e1}), we have
${r_{\ell}}(\sigma)\leq \beta_{M}\alpha_{\ell}<\infty$ for each $\ell\in\mathbb{N}$.
In the sequel, we show that $C_{\text{BSE}}(\Delta(\mathcal A))$
is a semisimple Fr\'echet algebra with respect to the seminorms $(r_{\ell})_{\ell}$
and the pointwise product which equals the first Arens product. Indeed, for $F_{1}$ 
and $F_{2}$ in ${\mathcal A}^{\ast\ast}$, we have ${\sigma_{1}}.\sigma_{2}=F_{1}\square F_{2}$
where $\sigma_{1}=F_{1}{\vert}_{\Delta(\mathcal A)}$ and 
$\sigma_{2}=F_{2}{\vert}_{\Delta(\mathcal A)}$.

\begin{theorem}
Let $(\mathcal A, p_\ell)$ be a commutative semisimple 
Fr\'echet algebra. Then,
$\big{(}C_{\text{BSE}}(\Delta(\mathcal A)),r_\ell\big{)}$ is a commutative semisimple 
Fr\'echet subalgebra. 
\end{theorem}

\begin{proof}
We first show that $C_{\text{BSE}}(\Delta(\mathcal A))$ is a linear subspace of
$C_b(\Delta(\mathcal A))$. Suppose that $\sigma_1, \sigma_2 \in
C_{\text{BSE}}(\Delta(\mathcal A))$. Then, there exist bounded sets $M_{1}$ and 
$M_{2}$ in $\mathcal A$ and positive real numbers $\beta_{M_{1}}$ 
and $\beta_{M_{2}}$ such that for every finite number
 of complex-numbers $c_1,\cdots, c_n$ and the same number of 
 $\varphi_1,\cdots, \varphi_n$ in $\Delta(\mathcal A)$ we have
 $$
 \vert\sum_{i=1}^{n}c_i\sigma_1(\varphi_i)\vert\leq\beta_{M_{1}} P_{M_{1}}(\sum_{i=1}^{n}c_i\varphi_i)\;\;\;\;\;\text{and}\;\;\;\;\;
 \vert\sum_{i=1}^{n}c_i\sigma_2(\varphi_i)\vert\leq\beta_{M_{2}}
P_{M_{2}}(\sum_{i=1}^{n}c_i\varphi_i).
 $$
 Therefore,
\begin{align*}
\vert\sum_{i=1}^{n} c_i(\sigma_1+\sigma_2)(\varphi_i)\vert&=\vert
\sum_{i=1}^{n}\big{(}c_i\sigma_1(\varphi_i)+c_i
\sigma_2(\varphi_i)\big{)}\vert\\
&\leq\vert\sum_{i=1}^{n}c_i\sigma_1(\varphi_i)\vert+\vert\sum_{i=1}^{n}c_i\sigma_2(\varphi_i)\vert\\
&\leq\beta_{M_{1}} P_{M_{1}}(\sum_{i=1}^{n}c_i\varphi_i)+\beta_{M_{2}}
P_{M_{2}}(\sum_{i=1}^{n}c_i\varphi_i)\\
&=\beta_{M}P_M(\sum_{i=1}^{n} c_i\varphi_i),
\end{align*}
where $M=M_{1} \cup M_{2}$ and $\beta_{M}=\beta_{M_{1}}+\beta_{M_{2}}$. 
Thus, $\sigma_{1}+\sigma_{2} \in C_{\text{BSE}}(\Delta(\mathcal A))$. 
We now show that the product of two functions in $C_{\text{BSE}}(\Delta(\mathcal A))$ 
is a {\text BSE}-function. Consider an arbitrary $\varepsilon>0$.
Due to the definition of $P_{M_{1}}$, we can 
choose $a_{1}\in M_{1}$ such that
\begin{center}
$P_{M_{1}}(\sum_{i=1}^{n} c_i\sigma_2(\varphi_i)\varphi_i)\leq
\vert \sum_{i=1}^{n}c_i
\sigma_2(\varphi_i)\varphi_i (a_1)\vert+\varepsilon$.
\end{center}
Moreover, we can choose $a_{2}\in\mathcal M_{2}$ such that 
\begin{center}
$P_{M_{2}}(\sum_{i=1}^{n} c_i\varphi_i(a_1)\varphi_i)\leq
\vert \sum_{i=1}^{n}c_i\varphi_i(a_1)\varphi_i (a_{2})
\vert+\varepsilon$.
\end{center}
Consequently,
\begin{align*}
\vert\sum_{i=1}^{n}c_i(\sigma_1\sigma_2)(\varphi_i)\vert
&\leq \beta_{M_{1}}P_{M_{1}}(\sum_{i=1}^{n}c_i\sigma_2(\varphi_i)\varphi_i)
\\
&\leq \beta_{M_{1}}\big{(}\vert \sum_{i=1}^{n}c_i\sigma_2(\varphi_i)\varphi_i (a_{1})
\vert+\varepsilon\big{)} \\
&\leq \beta_{M_{1}}\big{(}\beta_{M_{2}}P_{M_{2}}(\sum_{i=1}^{n}c_i\varphi_i (a_{1})\varphi_i)+\varepsilon\big{)} \\
&\leq  \beta_{M_{1}}\bigg{(}\beta_{M_{2}}\big{(}\vert \sum_{i=1}^{n}c_i
\varphi_i (a_{1})\varphi_i(a_2)\vert+\varepsilon \big{)}+\varepsilon\bigg{)} \\
&=\beta_{M_{1}}\bigg{(}\beta_{M_{2}}\big{(}\vert \sum_{i=1}^{n}c_i
\varphi_i (a_{1}a_2)\vert+\varepsilon \big{)}+\varepsilon\bigg{)}.
\end{align*}
We know from \cite{11} that 
$M:=M_{1}M_{2}$ is a bounded set in $\mathcal A$, and 
$$P_{M}(\sum_{i=1}^{n} c_i\varphi_i)=sup\big{\lbrace}\vert
\sum_{i=1}^{n}c_i\varphi_i(ab) \vert: a\in M_{1},\;b\in M_{2}\big{\rbrace}.
$$
Also, set $\beta_M:=\beta_{M_{1}}\beta_{M_{2}}$. Since $\varepsilon$ is arbitrary, by above inequalities, we have
$$
\vert\sum_{i=1}^{n}c_i(\sigma_1\sigma_2)(\varphi_i)\vert\leq \beta_MP_M(\sum_{i=1}^{n}c_i
\varphi_i).
$$
For this reason, $\sigma_1\sigma_2\in C_{\text{BSE}}(\Delta(\mathcal A))$.
Furthermore, it is obvious to see that 
$$r_{\ell}(\sigma_1\sigma_2)\leq r_{\ell}(\sigma_1)r_{\ell}(\sigma_2)\;\;\;\;\;\;\;\;\;\;(\ell\in\mathbb{N}),$$
which 
 implies that each $r_{\ell}$ is a submultiplicative seminorm. 
 
 To prove that $\big{(}C_{\text{BSE}}(\Delta(\mathcal A)),r_\ell\big{)}$ is a Fr\'echet algebra, it is enough to show that
 $\big{(}C_{\text{BSE}}(\Delta(\mathcal A)),r_\ell\big{)}$ is complete.
 For this reason, assume that 
$(\sigma_n)_n$ is a Cauchy sequence in $C_{\text{BSE}}(\Delta(\mathcal
A))$. Therefore, for each $\ell\in\mathbb{N}$ and $\varepsilon>0$, there exists
$N_\ell>0$ such that for all $m,n\geq N_\ell$, we have
$r_\ell(\sigma_n-\sigma_m)<\varepsilon$. 
Thus,
$$
\vert(\sigma_n-\sigma_m)(\varphi)\vert\leq r_\ell(\sigma_n-\sigma_m)<\varepsilon\;\;\;\;\;\;\;\;\;\;\;\big{(}\varphi\in \Delta(\mathcal A)\big{)}.
$$
Consequently, for each $\varphi\in \Delta(\mathcal A)$,
$\big{(}\sigma_n(\varphi)\big{)}_n$ is a Cauchy
sequence
in $\mathbb{C}$ and so is convergent to some complex-number.
We set $\sigma(\varphi):=\underset{n}{\lim}\;\sigma_n(\varphi)$, for 
 $\varphi$ in $\Delta(\mathcal A)$. 
 To show
 $\sigma\in C_{\text{BSE}}(\Delta(\mathcal A))$, 
 let $c_1,\cdots, c_k \in \CC $ and $\varphi_1,\cdots, \varphi_k$ in $\Delta(\mathcal A)$.
 Since each $\sigma_n$ belongs to $C_{\text{BSE}}(\Delta(\mathcal A))$,
 there exist a bounded set $M$ in $\mathcal A$ and 
 a positive real number $\beta_{M}$ such that the inequality
$$\vert\sum_{i=1}^{k} c_{i}\sigma_n(\varphi_{i})\vert\leq\beta_{M}
P_{M}(\sum_{i=1}^{k}c_{i}\varphi_{i})$$
holds. In addition, by definition of $r_\ell$, we have
$$\vert\sum_{i=1}^{k} c_i(\sigma_n-\sigma_m)(\varphi_i)\vert\leq r_\ell(\sigma_n-\sigma_m)<\varepsilon,$$
 for all $n,m\geq N_{\ell}$ and $P_{M_{\ell}}(\sum_{i=1}^{k}c_i\varphi_i)\leq 1.$
Therefore, by taking the limit with respect to $m$, we have
$$\vert\sum_{i=1}^{k} c_i(\sigma_n-\sigma)(\varphi_i)\vert\leq\varepsilon,$$
 for all $n\geq N_{\ell}$. Hence,
 $$
 \vert\sum_{i=1}^{k} c_{i}\sigma(\varphi_{i})\vert =
 \vert\sum_{i=1}^{k} c_{i}(\sigma_n-\sigma)(\varphi_i)\vert +
\vert\sum_{i=1}^{k} c_{i}\sigma_n(\varphi_i)\vert\leq \varepsilon + \beta_{M}
P_{M}(\sum_{i=1}^{n}c_{i}\varphi_{i}),
 $$
 which implies that 
 $\sigma\in C_{\text{BSE}}(\Delta(\mathcal A))$. Thus, 
$\big{(}C_{\text{BSE}}(\Delta(\mathcal A)),r_\ell\big{)}$ is a Fr\'echet algebra. 

For semisimplicity, note that
 for each $\varphi\in\Delta(\mathcal A)$,
the function $\Phi_{\varphi}$ defined as
$$\Phi_{\varphi}(\sigma):=\sigma(\varphi) \;\;\;\;\;\;\;\;\;\;\big{(}\sigma\in
C_{\text{BSE}}(\Delta(\mathcal A))\big{)},$$
 is a nonzero linear multiplicative
functional on $C_{\text{BSE}}(\Delta(\mathcal A))$. Consequently,
\begin{center}
$\underset{\Phi\in\Delta(C_{\text{BSE}}(\Delta (A))}{\bigcap}\ker
\Phi\subseteq \underset{\varphi\in\Delta(\mathcal A)}{\bigcap}
\ker \Phi_{\varphi}=\lbrace 0\rbrace$.
\end{center}
This inclusion implies that $C_{\text{BSE}}(\Delta(\mathcal A))$ is semisimple.
\end{proof}

\begin{lemma}
Let $(\mathcal A,p_\ell)$ be a commutative Fr\'echet algebra and
 $C_{\text{BSE}}(\Delta(\mathcal A))=C_{b}(\Delta(\mathcal A))$. 
 Then, for $\varphi_1,\cdots,\varphi_n\in \Delta(\mathcal A)$ and
$c_1,\cdots,c_n\in\mathbb{C}$, there exists $a\in \mathcal A$ such that 
$\widehat{a}(\varphi_i)=c_i$ for each $i=1,\cdots,n$.
\end{lemma}

\begin{proof}
Let $F=\lbrace\varphi_1,\cdots,\varphi_n\rbrace$. By Urysohn's lemma, there exists 
 $f_{i}\in C_{b}(\Delta(\mathcal A))$ in which $f_{i}(\varphi_{i})=1$ 
 and $f_{i}(\varphi_{j})=0$ for $i\neq j$. 
 Set $\sigma_{F}:=\sum_{i=1}^{n}c_{i}f_{i}$. Then,
$\sigma_{F} \in C_{b}(\Delta(\mathcal A))$ and
 $\sigma_{F}(\varphi_i)=c_i$ for each $i=1,\cdots,n$. 
 We now define the linear map
 $$T_{F}:\langle\varphi_1,\cdots,\varphi_n\rangle\rightarrow\mathbb{C}, \;\;\;\;
T_{F}(\sum_{i=1}^{n} d_{i}\varphi_{i})=\sum_{i=1}^{n} d_{i}c_{i}$$ 
where $d_1,\cdots,d_n\in\mathbb{C}$.
 Moreover, $\sigma_{F} \in C_{\text{BSE}}(\Delta(\mathcal A).$ Thus, there exist
 a bounded set $M$ in $\mathcal A$ and a positive real number
 $\beta_{M}$ such that
\begin{align*}
\vert T_{F}(\sum_{i=1}^{n} d_{i}\varphi_{i})\vert =
\vert \sum_{i=1}^{n} d_{i}\sigma_{F}(\varphi_{i})\vert
\leq\beta_{M}P_{M}(\sum_{i=1}^{n}d_{i}\varphi_{i}).
\end{align*}
Hence, $T_{F}$ is continuous. In addition, $\langle\varphi_1,\cdots,\varphi_n\rangle$
 is of finite dimension and so $T_{F}$ is weak$^{*}$-continuous. 
 Therefore by applying the general framework of the Hahn-Banach theorem \cite[Proposition 22.12]{11}, there exists 
 linear and weak$^{*}$-continuous function $T$ such that 
 $T{\vert}_{\langle\varphi_1,\cdots,\varphi_n\rangle}=T_{F}$. 
Thus,$$T(\varphi_i)=\sigma_{F}(\varphi_i)=c_{i}. $$ 
Furthermore, there exists $a\in \mathcal A$ such that
 $T=\widehat{a}$. Then, $\varphi_i(a)=\widehat{a}(\varphi_i)=c_{i}$.  
\end{proof}

\begin{proposition}\label{t1}
Let $(\mathcal A,p_\ell)$ be a commutative semisimple Fr\'echet algebra. Then, 
\begin{enumerate}
\item[(i)] $\widehat{a}\in C_{\text{BSE}}(\Delta(\mathcal A))$ for each $a\in \mathcal A.$
\item[(ii)] $C_{\text{BSE}}(\Delta(\mathcal A))$ equals the set of all
 $\sigma \in C_{b}(\Delta(\mathcal A))$ where 
there exists a bounded net $(a_{\lambda})_{\lambda}$ in $\mathcal A$
 with $\lim\widehat{a_{\lambda}}(\varphi)=\sigma(\varphi)$
 for all $\varphi \in \Delta(\mathcal A).$
\item[(iii)] $C_{\text{BSE}}(\Delta(\mathcal A))=\mathcal A^{**}
\vert_{\Delta(\mathcal A)}\bigcap C_{b}(\Delta(\mathcal A))$.
\end{enumerate}
\end{proposition}

\begin{proof} (i) 
Let $\varphi_1,\cdots,\varphi_n\in \Delta(\mathcal A)$ and
$c_1,\cdots,c_n\in\mathbb{C}$. 
Since $\sum_{i=1}^{n} c_i \varphi_i
\in{\mathcal A}^\ast$, following \cite{11},
there exist $K>0$ and a bounded set
$M\subseteq \mathcal A$ such that
\begin{center}
$\vert\sum_{i=1}^{n} c_i \varphi_i(a) \vert
\leq K P_M(\sum_{i=1}^{n} c_i  \varphi_i )\;\;\;\;\;\;\;\;\;\;\;\;(a\in\mathcal A)$.
\end{center}
Hence, $\vert\sum_{i=1}^{n} c_i \widehat{a}(\varphi_i)\vert
\leq K P_M(\sum_{i=1}^{n} c_i  \varphi_i )$, for each $a\in\mathcal A$.
Now, suppose that the net $(f_\alpha)_\alpha$ is convergent to $f$ in ${\mathcal
A}^\ast$ with respect to the strong topology.
Thus, for each bounded subset $M\subseteq\mathcal A$, we have
 $P_M (f_{\alpha} - f)\longrightarrow_\alpha 0.$ In particular,
 for $a\in\mathcal{A}$, set 
$M:=\lbrace a\rbrace$. Then,
$$
\vert\widehat{a}(f_{\alpha})-\widehat{a}(f)\vert=\vert f_{\alpha}(a)-f(a)\vert =P_M(f_{\alpha}-f)\longrightarrow_\alpha 0.
$$
It follows that $\widehat{a}$ is continuous, for each $a\in\mathcal{A}$. 
Therefore, $\widehat{a}$ is a {\text BSE}-function, and obviously $r_\ell
(\widehat{a})\leq K$, for each $a\in\mathcal{A}$. 

(ii) Let $\sigma \in C_{b}(\Delta(\mathcal A))$ be such that there exists 
a bounded net $(a_{\lambda})_{\lambda}$ in $\mathcal A$ 
with $\lim\widehat{a_{\lambda}}(\varphi)
=\sigma(\varphi)$, for all $\varphi\in \Delta(\mathcal A)$. To show $\sigma \in C_{\text{BSE}}(\Delta(\mathcal A))$,
 let $c_1,\cdots,c_n\in\mathbb{C}$ and $\varphi_1,\cdots,\varphi_n\in \Delta(\mathcal A)$. Then, for arbitrary $\varepsilon>0$,
  we have  
\begin{align*}
\vert \sum_{i=1}^{n}c_i\sigma (\varphi_i
)\vert&\leq\vert\sum_{i=1}^{n}c_i\widehat{a_\lambda}
(\varphi_i)\vert+\vert\sum_{i=1}^{n}
c_i(\widehat{a_\lambda}(\varphi_i)-\sigma (\varphi_i ))\vert\\
&\leq \beta_{M}P_{M}(\sum_{i=1}^{n}c_i
\varphi_i)+\varepsilon,
\end{align*}
where $M$ is a bounded set in $\mathcal A$ and $\beta_{M}$ is a positive real number.
Since $\varepsilon$ is arbitrary, $\sigma\in C_{\text{BSE}}(\Delta(\mathcal A))$. 

Conversely, let 
$\sigma\in C_{\text{BSE}}(\Delta(\mathcal A))$. Since $\Delta(\mathcal A)$ endowed
with the weak$^*$ topology is a locally convex space, there exists a linear map
$\overline{\sigma}:span(\Delta(\mathcal A)) \longrightarrow
\mathbb{C}$ such that
$\overline{\sigma}(\varphi)=\sigma(\varphi)$ for each
$\varphi\in\Delta(\mathcal A)$; see \cite{9-1} and  
also see \cite[Proposition 22.12]{11}, the general framework for the Hahn-Banach theorem and its consequences.
On the other hand, there
exist a bounded set $M$ of $\mathcal A$ and 
 a positive real number $\beta_{M}$ such that for every finite number
 of complex-numbers $c_1,\cdots, c_n$ and the same number of 
 $\varphi_1,\cdots, \varphi_n$ in $\Delta(\mathcal A)$ the inequality
$$\vert\sum_{i=1}^{n} c_{i}\sigma(\varphi_{i})\vert\leq\beta_{M}
P_{M}(\sum_{i=1}^{n}c_{i}\varphi_{i})$$
holds. 
Note that if $\sum_{i=1}^{n}c_i\varphi_i=0$, then $\sum_{i=1}^{n}
c_i\sigma(\varphi_i)=0.$ Consequently, $\overline{\sigma}$ is
well-defined. Moreover, $\vert\overline{\sigma}(f)\vert\leq\beta_M
\,P_{M}(f)$ for $f \in span(\Delta(\mathcal A))$. In addition, since $\overline{\sigma}$ is linear and
weak$^*$-continuous on the closure of
 $span(\Delta(\mathcal A))$, using again \cite[Proposition 22.12]{11},
$\overline{\sigma}$ has a unique
weak$^{*}$-continuous extension $\overline{F}$ to
${\mathcal A}^{*}$ such that $\vert \overline{F}(f)\vert\leq\beta_M
P_{M}(f)$ for all $f\in {\mathcal A}^*$. Let 
$F\in w^\ast\text{-}cl(\overline{F})$. Then, there exists a bounded net $(a_\lambda)_{\lambda}$ in $\mathcal A$ 
where $\widehat{a_\lambda}\overset{w^\ast}{\longrightarrow}_\lambda F$. Thus,
$$\widehat{a_\lambda}(\varphi)\longrightarrow_\lambda F(\varphi)=
\overline{\sigma}(\varphi)=\sigma(\varphi)\quad\;\;\;\;(\varphi\in\Delta(\mathcal A)).$$

(iii) Let $\sigma\in\mathcal A^{**}\vert_{\Delta(\mathcal A)}\bigcap C_{b}(\Delta(\mathcal A))$.  
Following \cite{11}, there exist $K>0$ and a bounded set 
$M$ of $\mathcal A$ such that $\vert \sigma(\varphi)\vert\leq K P_{M}(\varphi)$ for each $\varphi\in {\mathcal A}^\ast$. In
 particular,
 for every finite number
 of complex-numbers $c_1,\cdots, c_n$ and the same number of 
 $\varphi_1,\cdots, \varphi_n$ in $\Delta(\mathcal A)$
 we have
$$
\vert \sum_{i=1}^{n}c_i \sigma(\varphi_i)\vert  
\leq K P_{M}(\sum_{i=1}^{n}c_i(\varphi_i)).
$$
Therefore, $\sigma \in C_{\text{BSE}}(\Delta(\mathcal A))$. 

Conversely,
 let $\sigma \in C_{\text{BSE}}(\Delta(\mathcal A))$, and suppose that
 $\varphi_1,\cdots, \varphi_n\in\Delta(\mathcal A)$. Set $V:=\langle\varphi_1,\cdots, \varphi_n\rangle$.
 Consider the linear function 
 $$F _{V}:V\rightarrow \mathbb{C},\;\; \;\;
 F _{V}(\sum_{i=1}^{n}c_i \varphi_i)=\sum_{i=1}^{n}c_i \sigma(\varphi_i)$$
 where $c_1,\cdots,c_n\in\mathbb{C}$.
 Therefore, by assumption, there exist a bounded set $M$ in $\mathcal A$ and a positive real number $\beta_M$
 such that
 $$\vert F _{V}(\sum_{i=1}^{n}c_i \varphi_i)\vert \leq 
\beta_{M}P_{M}(\sum_{i=1}^{n}c_i \varphi_i).$$
Thus, 
$F _{V}$ is continuous and 
 there exists a weak$^{*}$-continuous extension $\overline{F} _{V}\in{\mathcal A}^{**}$, where
 $\vert \overline{F} _{V}(f)\vert \leq \beta_{M}P_{M}(f)$ for each $f\in{\mathcal A}^{*}$. 
 If now $\Psi\in w^{*}\text{-}cl({\overline{F} _V})$, then 
 there exists a net $(\Psi_{\alpha})_{\alpha}\subseteq \overline{F} _V$
 such that $\Psi_{\alpha}\overset{w^\ast}{\longrightarrow}_{\alpha} \Psi$.
 Consequently, there exists $\alpha_{0}$ 
 such that for each $\alpha\geq\alpha_{0}$, we have
 $$\vert\Psi_{\alpha}(f)-\Psi(f)\vert<\varepsilon\;\;\;\;\;\;\;\;\;\;\;(f\in{\mathcal A}^{*}).$$ 
In particular, if $f=\sum_{i=1}^{n}c_i \varphi_i$ for $c_1,\cdots,c_n\in\mathbb{C}$,
 then  
 $$\vert\sum_{i=1}^{n}c_i \sigma(\varphi_i)-\Psi(\sum_{i=1}^{n}c_i \varphi_i)\vert<\varepsilon.$$
Specifically, $\Psi(\varphi)=\sigma(\varphi)$ for $\varphi\in \Delta(\mathcal A).$ On the other hand, 
$\sigma\in C_{b}(\Delta(\mathcal A))$, and so 
 $\sigma\in \mathcal A^{**}\vert_{\Delta(\mathcal A)}\bigcap C_{b}(\Delta(\mathcal A)).$
\end{proof}

\section{The multiplier algebra and {\text BSE}-Fr\'echet algebras}

Consider the Fr\'echet algebra $(\mathcal{A},p_\ell)$.
Let us first recall from \cite{10} that
the map $T:\mathcal
A\rightarrow\mathcal A$ is called a multiplier on $\mathcal A$ if
$T(a)b=aT(b)$ for all $a,b\in \mathcal A$. The set of all
multipliers on $\mathcal A$ will be denoted by $M(\mathcal A)$.
Let now
\begin{center}
$B(\mathcal A)=\lbrace T:\mathcal A\longrightarrow \mathcal A
\,\,\vert \,\,T \,\,is \,\, linear \,\, and \,\,
continuous\rbrace$.
\end{center}
The strong operator topology on $B(\mathcal
A)$ is generated by the family of seminorms $(
q_{\ell,a})$, defined as
 $q_{\ell,a}(T):=p_\ell(Ta)$
for all $a\in\mathcal A$ and $\ell\in \mathbb{N}$.
The following result is a generalization of \cite[Theorem
1.1.2]{10}.

\begin{proposition}
Let $(\mathcal A,p_\ell)$ be a commutative semisimple Fr\'echet algebra.
 Then, $M(\mathcal A)$ endowed with the
strong operator topology, is a unital and commutative complete
locally convex subalgebra of $B(\mathcal A)$.
\end{proposition}

\begin{proof}
Suppose that $T\in M(\mathcal A)$. Similar to the proof of
\cite[Theorem 1.1.2]{10}, one can show that $T$ is linear. Now, we
show that $T$ is continuous. Let $(b_{\alpha})_{\alpha}$ be a net in $\mathcal A$,
converging to $b\in\mathcal A$ and suppose that $(Tb_{\alpha})_{\alpha}$
converges to $c\in \mathcal A$. We show that $Tb=c$. For all $a\in \mathcal A$ and $\ell\in\mathbb{N}$, we
have
\begin{align*}
p_\ell \big{(}ca-(Tb)a\big{)}&=p_\ell\big{(}ca-(Tb_\alpha)a+(Tb_\alpha)a-(Tb)a\big{)}\\
&\leq p_\ell\big{(}ca-(T b_\alpha)a\big{)}+p_\ell\big{(}b_\alpha (Ta)-b (Ta)\big{)}\\ 
&\leq
p_\ell(c-Tb_\alpha)p_\ell(a)+p_\ell(b_\alpha-b)p_\ell(Ta)\\
&\longrightarrow 0.
\end{align*}
Hence, $p_\ell\big{(}ca-(Tb)a\big{)}=0$ for all $\ell\in\mathbb{N}$. Thus,
\begin{center}
$ca-(Tb)a\in\underset{\ell}{\bigcap}\;\text{ker}\,p_\ell=\lbrace 0\rbrace$.
\end{center}
Consequently, by Lemma \ref{withoutorder},
$c-Tb=0$ and so $Tb=c$
as claimed. Similar to the Banach case, one can prove that
$M(\mathcal A)$ is a unital and commutative subalgebra of
$B(\mathcal A)$. We only show that $M(\mathcal A)$ is closed in
$B(\mathcal A)$, with respect to the strong operative topology.
Let $(T_{\alpha})_{\alpha}$ be a net in $M(\mathcal A)$ such that
$T_{\alpha}\longrightarrow_{\alpha} T$ under the strong operator topology.
 Then,
\begin{center}
$q_{\ell,a}(T_{\alpha}-T)=p_\ell(T_{\alpha}a-Ta)\longrightarrow 0$
\end{center}
for all $a \in \mathcal A$ and $\ell\in \mathbb{N}$. Thus,  
$T_{\alpha}a\longrightarrow_{\alpha} Ta$ in $\mathcal A$. It
follows that $T_{\alpha}(ab)\longrightarrow_{\alpha} T(ab)$ 
for all $a,b\in\mathcal A$.
Also,
$aT_{\alpha}(b)\longrightarrow_{\alpha} aT(b)$ for all $a,b\in\mathcal A$, since the multiplication in $\mathcal A$ is continuous.
 Thus, for each $a,b\in\mathcal A$ we have 
  $T(ab)=aT(b)$ which implies $T\in M(\mathcal A)$.
\end{proof}

We state here a theorem, which leads us
to introduce {\text BSE}-Fr\'echet algebras. 
The proof is similar to the Banach case \cite[Theorem 1.2.2]{10}, and so it is omitted.

\begin{theorem}\label{ThatT}
Let $(\mathcal A,p_\ell)$ be a commutative semisimple 
Fr\'echet algebra and $T\in M(\mathcal A)$. Then, there
exists a unique continuous and bounded function $\widehat{T}$ on
$\Delta(\mathcal A)$ such that
$\widehat{T(a)}=\widehat{T}\;\widehat{a}$ for each $a\in\mathcal A$.
\end{theorem}

Note that in the proof of Theorem \ref{ThatT}, $\widehat{T}$ is defined by
$$
\widehat{T}(\varphi)=\frac{\varphi\big{(}T(a)\big{)}}{\varphi(a)}\;\;\;\;\;\;\;\;\;\;\big{(}\varphi\in\Delta(\mathcal{A})\big{)}
$$
for some $a\in\mathcal{A}$ with $\varphi(a)\not=0$. Clearly, this definition of $\widehat{T}$ is independent of the choice of $a$.

\begin{remark}\label{t5} 
Following \cite{10}, $\widehat{M(\mathcal A)}$ is always unital.
Consider the identity map $I: \mathcal
A\longrightarrow\mathcal A$ and let $\varphi\in\Delta(\mathcal{A})$.
Then, for some $a\in\mathcal{A}$ with $\varphi(a)\not=0$, we have
\begin{center}
$\widehat{I}(\varphi)=\frac{\varphi (I(a))}{\varphi(a)}=
\frac{\varphi(a)}{\varphi(a)}=1$,
\end{center}
where $1$ is
the constant function.
\end{remark}

Now, we are in a position to introduce the concept of {\text BSE}-Fr\'echet
algebra.

\begin{definition}\rm
Let $(\mathcal A,p_\ell)$ be a commutative semisimple Fr\'echet algebra. 
Then, $\mathcal A$ is called a
{\text BSE}-Fr\'echet algebra if
$
C_{\text{BSE}}(\Delta(\mathcal A))=\widehat{M(\mathcal A)}.
$
\end{definition}

Following the definition of a bounded $\Delta$-weak approximate
identity in the sense of Jones-Lahr \cite{J} and also 
\cite{D}, a net $(e_{\lambda})_{\lambda}$ in the Fr\'echet
algebra $(\mathcal A,p_\ell)$ is called a $\Delta$-weak approximate
identity if $\varphi(e_{\lambda})\longrightarrow_{\lambda} 1$ for all
$\varphi\in\Delta(\mathcal A)$, or equivalently
$\varphi(e_\lambda\,a)\longrightarrow\varphi(a)$ for all
$a\in\mathcal A$ and $\varphi\in\Delta(\mathcal A)$. Moreover,
$(e_{\lambda})_{\lambda}$ is called bounded $\Delta$-weak
approximate identity if $(e_{\lambda})_{\lambda}$ is a bounded
subset of $\mathcal A$.

The following result is a generalization of \cite[Corollary 5]{15}, which implies that all
{\text BSE}-Fr\'echet algebras have bounded $\Delta$-weak approximate
identity.

\begin{theorem}\label{t2}
Let $(\mathcal A,p_\ell)$ be a commutative semisimple Fr\'echet algebra. 
Then, $\widehat{M(\mathcal A)}\subseteq
C_{\text{BSE}}(\Delta(\mathcal A))$ if and only if $\mathcal A$ has a
bounded $\Delta$-weak approximate identity.
\end{theorem}

\begin{proof}
First, suppose that $\mathcal A$ has a bounded $\Delta$-weak
approximate identity $(e_\lambda)_{\lambda}$ and $T\in M(\mathcal
A)$. For $c_1,\cdots,c_n \in\mathbb{C}$, 
$\varphi_{1},\cdots,\varphi_{n}\in\Delta(\mathcal A)$ and each $\varepsilon> 0$, there
exists $\lambda_0$ such that for all $\lambda\geq\lambda_0$
\begin{align*}
\vert\sum_{i=1}^{n} c_i\widehat{T}(\varphi_i)\vert&\leq\vert
\sum_{i=1}^{n}c_i( \widehat{T}(\varphi_i)-\widehat{T}
(\varphi_i)\widehat{e_{\lambda}}(\varphi_i))\vert+\vert
\sum_{i=1}^{n}c_i\widehat{T}(\varphi_i )
\widehat{e_{\lambda}}(\varphi_i)\vert\\
&=\vert\sum_{i=1}^{n}c_i\widehat{T}(\varphi_i)(1-
\widehat{e_{\lambda}}(\varphi_i))\vert+\vert\sum_{i=1}^{n}
c_i\widehat{T (e_{\lambda})}(\varphi_i)\vert\\
&\leq\varepsilon +\vert\sum_{i=1}^{n}
c_i\widehat{T (e_{\lambda})}(\varphi_i)\vert.
\end{align*}
According to part (i) of the Proposition \ref{t1}, 
$\widehat{T(e_{\lambda})}$ is a {\text BSE}-function. 
Then, there exist a bounded set $M$ in $\mathcal A$ and positive
 real number $\beta_{M}$ where $$
 \vert\sum_{i=1}^{n}
c_i\widehat{T (e_{\lambda})}(\varphi_i)\vert \leq \beta_{M}P_{M}
(\sum_{i=1}^{n}c_i\varphi_i).$$
Since $\varepsilon$ is arbitrary, we have
$$
\vert \sum_{i=1}^{n}c_i\widehat{T}(\varphi_i)\vert \leq
\beta_{M}P_{M}(\sum_{i=1}^{n}c_i\varphi_i).
$$
Hence, $\widehat{T}\in C_{\text{BSE}}(\Delta(\mathcal A))$.

Conversely, suppose that $\widehat{M(\mathcal A)}\subseteq
C_{\text{BSE}}(\Delta(\mathcal A))$. 
By applying 
part (ii) of the Proposition \ref{t1}, for
the constant function $1\in C_{\text{BSE}}(\Delta(\mathcal A))$, there exists a
bounded net $(a_\lambda)_{\lambda}$ in $\mathcal A$ such that
$\varphi(a_\lambda)\longrightarrow_\lambda 1$ ($\varphi\in\Delta(\mathcal
A))$. It follows that $(a_\lambda)_{\lambda}$ is a bounded
$\Delta$-weak approximate identity.
\end{proof}

Before proving the last result of the paper,
let $\sigma\in C_{\text{BSE}}(\Delta(\mathcal A))$ and
$$P_{M}(\sigma)=\inf\big{\lbrace}\beta_{M}:
\vert\sum_{i=1}^{n}c_{i}\sigma(\varphi_i)\vert\leq\beta_{M}
P_{M}(\sum_{i=1}^{n} c_{i} \varphi_i ),\; c_i
 \in \mathbb{C},\; \varphi_i \in
\Delta({\mathcal A}) \big{\rbrace}.$$
Then, $\big{(}C_{\text{BSE}}(\Delta(\mathcal A)),P_{M}\big{)}$ is a locally convex algebra. In addition, 
we define 
$$\beta_{M}({\mathcal A})=\sup\big{\lbrace}\vert\sum_{i=1}^{n}c_{i}\vert:
P_{M}(\sum_{i=1}^{n} c_{i} \varphi_i ) \leq 1,\;c_1,\cdots, c_n
 \in \mathbb{C},\;\varphi_1 ,\cdots , \varphi_n \in
\Delta({\mathcal A})\big{\rbrace},$$ 
where $M$ is a bounded set in
 $\mathcal A$. It is easy to prove that
 $$
P_{M}(\sigma)=\sup\big{\lbrace}\vert\sum_{i=1}^{n}c_{i}\sigma(\varphi_i)\vert:
 P_{M}(\sum_{i=1}^{n} c_{i} \varphi_i )\leq 1,\; c_i
 \in \mathbb{C},\; \varphi_i \in
\Delta({\mathcal A})\big{\rbrace}.
 $$
We note that if $C_{\text{BSE}}(\Delta(\mathcal A))$ is unital, then 
$\beta_{M}(\mathcal A)=P_{M}(1)$ where $1$ is
the constant function.

\begin{theorem}
Let $(\mathcal A,p_\ell)$ be a commutative semisimple Fr\'echet algebra.
 Then, the following statements are
equivalent.
\begin{enumerate}
\item[(i)] $C_{\text{BSE}}(\Delta(\mathcal A))$ is unital; \item[(ii)]
$\mathcal A$ has a bounded $\Delta$-weak approximate identity;
\item[(iii)] $\beta_{M}(\mathcal A)<\infty$, for some bounded set $M$ in $\mathcal A$.
\end{enumerate}
\end{theorem}

\begin{proof}
(i) $\Leftrightarrow$ (ii). Suppose that $C_{\text{BSE}}(\Delta (\mathcal
A)) $ is unital. By using part (ii) of the Proposition \ref{t1}, there exists a bounded net
$(a_{\lambda})_{\lambda}\subseteq \mathcal A$ which 
$\varphi(a_\lambda)=\widehat{a_\lambda}(\varphi)\longrightarrow_{\lambda}
1$ for each $\varphi \in \Delta(\mathcal A)$. This is equivalent
to saying that
the net $(a_\lambda)_\lambda$ is a bounded 
$\Delta$-weak approximate identity for $\mathcal A$. Hence, $(ii)$ 
is obtained. 

Conversely, let
 $\mathcal A$ has a bounded $\Delta$-weak approximate identity.
 Then, 
$\widehat{M(\mathcal A)} \subseteq C_{\text{BSE}}(\Delta(\mathcal A))$, by 
 Theorem \ref{t2}. Furthermore, the identity map $I: \mathcal A\rightarrow
\mathcal A$ belongs to $M(\mathcal A)$ and $\widehat{I}$ is the constant 
function $1$ on $\Delta(\mathcal A)$; see Remark \ref{t5}. Therefore, $1\in
C_{\text{BSE}}(\Delta(\mathcal A))$ and so 
$C_{\text{BSE}}(\Delta(\mathcal A))$ is unital.

(i) $\Leftrightarrow$ (iii). Let $C_{\text{BSE}}(\Delta(\mathcal A))$ be
unital. Then, there exist a bounded set $M$ in $\mathcal A$ and 
 a positive real number $\beta_{M}$ such that for every finite number
 of complex-numbers $c_1,\cdots, c_n$ and the same number of 
 $\varphi_1,\cdots, \varphi_n$ in $\Delta(\mathcal A)$ we have
$$
|\sum_{i=1}^{n}c_i|=\vert\sum_{i=1}^{n}c_i 1(\varphi_i)\vert\leq \beta_{M}
P_{M}(\sum_{i=1}^{n}c_i\varphi_i).
$$
Consequently, 
$\beta_{M}(\mathcal A)\leq \beta_{M}<\infty.$ 
Hence, (iii) is obtained. 

Conversely, suppose that (iii) holds and
$d:=P_{M}(\sum_{i=1}^{n} c_{i} \varphi_i )$
for $c_1,\cdots, c_n\in\mathbb{C}$ and $\varphi_1,\cdots, \varphi_n\in\Delta(\mathcal A)$.
Thus, 
$$P_{M}(\sum_{i=1}^{n} (\frac{c_{i}}{d})\varphi_i )=(\frac{1}{d})P_{M}(\sum_{i=1}^{n} c_{i} \varphi_i )=1,$$
and hence,
$$  
\vert\sum_{i=1}^{n}c_{i}1(\varphi_i)\vert\leq
\beta_{M}(\mathcal A)P_{M}(\sum_{i=1}^{n} c_{i} \varphi_i ).
$$
It follows that 
$C_{\text{BSE}}(\Delta(\mathcal A))$ is unital.
\end{proof}

{\bf Acknowledgment.} The authors would like to thank the University of
Isfahan for their support and to Dr. F. Abtahi for her useful suggestions and comments.

\vspace{9mm}

{\footnotesize \noindent

\noindent
 M. Amiri\\
Department of Mathematics,
   University of Isfahan,
    Isfahan, Iran\\
     mitra.amiri@sci.ui.ac.ir\\

\noindent
 A. Rejali\\
Department of Mathematics,
   University of Isfahan,
    Isfahan, Iran\\
    rejali@sci.ui.ac.ir\\

\end{document}